\documentclass[11pt, leqno]{amsart}
\usepackage{amsmath}
\usepackage{amsthm}
\usepackage{newtxtext, newtxmath}
\usepackage{mathrsfs}

\usepackage[colorinlistoftodos]{todonotes}
\renewcommand{\nu}{\upsilon}

\usepackage{graphicx}
\usepackage[font=small,labelfont=bf]{caption}
\usepackage{epstopdf}
\usepackage{xcolor}
\usepackage{float}
\usepackage{pgfplots}
\usepackage{listings}
\usepackage{longtable}
\usepackage{comment}
\usepackage{esint}
\usepackage{nicefrac}
\usepackage{tikz}
\usetikzlibrary{arrows, patterns, patterns.meta, calc, math, decorations.markings}
\usepackage[normalem]{ulem} 
\usepackage[parfill]{parskip}
\usepackage{esint}
\usepackage{linegoal}

\usepackage{mathtools}
\usepackage[utf8]{inputenc}
\usepackage[T1]{fontenc}
\usepackage[shortlabels, inline]{enumitem}
\DeclarePairedDelimiterX\set[1]\lbrace\rbrace{#1}

\usepackage{scalerel}[2014/03/10]
\usepackage[usestackEOL]{stackengine}

\usepackage[backref=page]{hyperref}
\hypersetup{
plainpages=false,
colorlinks,
linkcolor={cyan!90!black},
citecolor={magenta},
urlcolor={red!90!black},
bookmarksdepth=2,} 
\usepackage[noadjust]{cite}

\newtheorem{thm}{Theorem}[section]
\newtheorem{cor}[thm]{Corollary}
\newtheorem{prop}[thm]{Proposition}
\newtheorem{lem}[thm]{Lemma}
\newtheorem{qual-ass}[thm]{Qualitative Assumption}

\theoremstyle{definition}
\newtheorem{defn}[thm]{Definition}

\theoremstyle{remark}
\newtheorem{rem}[thm]{Remark}

\renewcommand{\L}{\mathrm{L}}
\newcommand{\Cont}{\mathrm{C}}

\newcommand{\e}{\mathrm{e}}
\renewcommand{\d}{\mathrm{d}}
\newcommand{\ii}{\mathrm{i}}
\newcommand{\dd}{\,\mathrm{d}}
\newcommand{\R}{\mathbb{R}}
\newcommand{\C}{\mathbb{C}}
\newcommand{\N}{\mathbb{N}}
\renewcommand{\S}{\mathrm{S}}
\newcommand{\one}{\boldsymbol{1}}
\newcommand{\eps}{\varepsilon}
\renewcommand{\Re}{\operatorname{Re}}
\renewcommand{\Im}{\operatorname{Im}}
\DeclareMathOperator{\Div}{div}
\DeclareMathOperator{\dom}{dom}
\renewcommand{\MR}{\operatorname{MR}}

\colorlet{shadecolor}{gray!20}
\pgfplotsset{compat=1.9}
\usepgflibrary{fpu}
\makeatletter
\let\c@equation\c@thm
\makeatother
\numberwithin{equation}{section}

\author{Moritz Egert}
\address{Department of Mathematics, Technical University of Darmstadt, Schlossgartenstr.\ 7, 64289 Darmstadt, Germany}
\email{egert@mathematik.tu-darmstadt.de}
\author{Jan Giesselmann}
\address{Department of Mathematics, Technical University of Darmstadt, Dolivostr.\ 15, 64293 Darmstadt, Germany}
\email{giesselmann@mathematik.tu-darmstadt.de}
\author{Marc Hoffmann}
\email{hoffmann@mathematik.tu-darmstadt.de}
\date{\today}
\title{Quantitative maximal $\L^p$-regularity for generators of real contraction semigroups}

\begin{document}
\begin{abstract} 
Let $-A$ be the generator of a real, symmetric contraction semigroup on a $\sigma$-finite measure space. We prove a universal bound for maximal $\L^p$-regularity: In space-time norms
\begin{align*}
    \|Au\|_p \leq \big(p^* + \tfrac{p^*}{2(p^*-1)}\big) \|f\|_p,
\end{align*}
where $p^* = \max(p,\nicefrac{p}{p-1})$ and $u$ is the mild solution to the Cauchy problem with right-hand side $f$ starting at $u(0)=0$. Linear growth of the constant in $p^*$ is optimal. Our proof introduces an elementary approach to maximal regularity that avoids both vector-valued harmonic analysis and dilation theory. We also obtain quantitative maximal regularity estimates for real, non-symmetric semigroups that are contractive on a sector.
\end{abstract}

\maketitle

\section{Introduction}
\label{sec:intro}
\subsection{Context}

Let $(\Omega,\mu)$ be a $\sigma$-finite measure space, and write $\L^p \coloneqq \L^p(\Omega,\mu)$ for the complex Lebesgue spaces. Suppose that $-A$ generates a bounded analytic $C_0$-semigroup $(T_t)_{t\geq0}$ on $\L^2$ which is real and absolutely contractive. Here, real means that each $T_t$ preserves real-valued functions, and absolute contractivity is the property that
\begin{equation} 
\label{eq:contraction property} 
\|T_t f\|_p \leq \|f\|_p, \qquad f\in\L^2\cap\L^p, 
\end{equation} 
for every $t\geq0$ and $1\leq p\leq\infty$. For $1<p<\infty$, we use the same notation for the consistent $\L^p$-realization of the semigroup and its generator. We call such a semigroup real contraction semigroup on $(\Omega, \mu)$ for short. This includes the symmetric submarkovian semigroups~\cite{Stein1970, Ouhabaz2005}. Symmetry means that each $T_t$ is symmetric on $\L^2$, or, equivalently, that $A$ is self-adjoint. In 1987, Lamberton~\cite{Lamberton1987} proved that the abstract
Cauchy problem
\begin{align}
\label{eq:ACP}
    u^\prime(t)+Au(t)=f(t), \qquad t>0, \qquad u(0)=0,
\end{align}
has maximal $\L^p$-regularity for every $1<p<\infty$. The mild solution is given by
\begin{align*}
    u(t)\coloneqq\int_0^t T_{t-s}f(s) \dd s,
    \qquad t\geq 0.
\end{align*}
Maximal regularity means that, for each such $p$, there is a constant $C$ such that, for every $f\in\L^p((0,\infty);\L^p)$, this solution belongs to $\dom(A)$ for almost every $t>0$ and satisfies \begin{equation} 
\label{eq:MR} \biggl(\int_0^\infty \|Au(t)\|_p^p \dd t\biggr)^{1/p} \leq C \biggl(\int_0^\infty \|f(t)\|_p^p \dd t\biggr)^{1/p}. 
\end{equation} 
For background on maximal regularity and its applications
to nonlinear evolution equations, we refer to
\cite{HvNVW3, DHP, KW}. 

We denote the best constant in \eqref{eq:MR} by $\MR_p(A)$ and write $p^* \coloneqq \max\{p, \nicefrac{p}{p-1}\}$ for $p>1$. 
Very little is known about the size of $\MR_p(A)$ and its growth as $p^*\to\infty$. One quantitative result is due to Hyt\"onen~\cite{Hytonen}: for the Poisson operator $A=\smash{\sqrt{-\Delta}}$ on $\R$, he obtained \begin{equation} \label{eq:Hytonen} \frac{p^*-1}{2} \leq \MR_p(\sqrt{-\Delta}) \leq p^*-1+\cot\Bigl(\frac{\pi}{2p^*}\Bigr). \end{equation} Here we have substituted the exact values of the constants $\beta_{p,\L^p}$ and $\hbar_{p,\L^p}$ appearing in~\cite{Hytonen}; see, for example,~\cite{HvNVW1}.
\subsection{Main results}

Our first main result gives a fully explicit estimate, uniform over all real, symmetric contraction semigroups.
\begin{thm}
\label{thm:symm}
Let $-A$ generate a real, symmetric contraction semigroup on a $\sigma$-finite measure space $(\Omega,
\mu)$, and let $1<p<\infty$. Then 
\begin{align*}
    \MR_p(A) \leq p^* + \frac{p^*}{2(p^*-1)}.
\end{align*}
\end{thm}
In particular, $\MR_p(A) \leq p^* + \nicefrac{1}{2} + o(1)$ as $p^* \to \infty$, which also improves the asymptotic upper bound for the Poisson semigroup in \eqref{eq:Hytonen}. At the same time, \eqref{eq:Hytonen} shows that the linear dependence on $p^*$ is best possible. Theorem~\ref{thm:symm} does not recover the bound $\MR_2(A) \leq 1$ that can be obtained from Plancherel's  theorem; see Section~\ref{sec:L2}. Thus, the inequality could be improved near $p=2$ via complex interpolation.

Universal inequalities for symmetric contraction semigroups have a long history, of which Theorem~\ref{thm:symm} is a further instance. We mention, among others, (i) the ergodic theorem of Dunford and Schwartz~\cite[Theorem VIII.7.7]{DS}, (ii) Liskevich and Perelmuter's formula for the angle of analyticity~\cite{Liskevich-Perelmuter}, see also \cite{Kriegler}, and (iii) the multiplier theorem of Carbonaro and Dragi\v{c}evi\'c~\cite{Carbonaro-Dragicevic}. While maximal regularity can be inferred from either (i)~\cite[Theorem~5.6]{KW} or (iii)~\cite{Dore-Venni}, the resulting bounds for $\MR_p(A)$ are highly suboptimal due to the vector-valued inequalities used as intermediate steps. 

Symmetric contraction semigroups are sectorially contractive. This means that
for every $p \in (1,\infty)$ there is an angle $\theta_p > 0$ such that $(T_t)_{t \geq 0}$ has an analytic extension to a contraction semigroup on $\L^p$ for complex times $t \in \S_{\theta_p}$. Here, $\S_\theta \subseteq \C$ is the open sector with opening half-angle $\theta$ symmetric about $(0,\infty)$. It is known that $\theta_p = \arccos|1- \nicefrac{2}{p}| > 0$ is optimal across the full class of symmetric contraction semigroups~\cite{Liskevich-Perelmuter}. In the following we only treat angles $\theta_p \leq \arccos|1- \nicefrac{2}{p}|$.
Our second main result applies to real, possibly non-symmetric sectorially contractive semigroups. This includes semigroups generated by divergence form operators; see Corollary~\ref{cor:Riesz-Thorin non-symmetric div} below.
\begin{thm}
    \label{thm:non-symm}
Let $-A$ generate a real, sectorially contractive semigroup on a $\sigma$-finite measure space $(\Omega,
\mu)$ and let $p$ be an even integer. Then 
\begin{align*}
    \MR_p(A) \leq 1 + (p-1) \cot(\theta_p).
\end{align*}
\end{thm}
The role of $p$ being an even integer will become clear below. Interpolation and duality yield the following quantitative estimates for all $p \in (1,\infty)$; see Section~\ref{sec:interpolation}. The first one relies on the generic bound for $\theta_p$ obtained by Stein interpolation~\cite[Theorem~10.10]{AVV26}, whereas the second one uses an improvement specific to divergence-form operators~\cite{Chill2021, Chill2006}. For the precise definition of these operators as generators of contraction semigroups we also refer to \cite{Chill2021}.
\begin{cor}
\label{cor:Riesz-Thorin non-symmetric}
Let $-A$ generate a real, sectorially contractive semigroup on a $\sigma$-finite measure space $(\Omega,
\mu)$, and let $1<p<\infty$. Then 
\begin{align*}
    \MR_p(A) \leq 1+ \frac{(p^*)^2}{2 \theta_2}.
\end{align*}
\end{cor}
\begin{cor}
\label{cor:Riesz-Thorin non-symmetric div}
Let $A = -\Div(\mu \nabla \cdot)$ be a divergence form operator with real, bounded, elliptic coefficients $\mu$ and mixed boundary conditions on an open set $\Omega \subseteq \R^n$. Let $1<p<\infty$ and let $\vartheta \in [0, \nicefrac{\pi}{2})$ be such that
\begin{align*}
    \bigl|\arg(\mu(x) \xi \cdot \overline{\xi})\bigr| \leq \vartheta, \qquad x \in \Omega,\, \xi \in \C^n \setminus \{0\}.
\end{align*}
Then 
\begin{align*}
    \MR_p(A) \leq 1+ \frac{\sqrt{1+\tan^2(\vartheta)}}{2}(p^*)^{3/2}.
\end{align*}
\end{cor}
The application to divergence-form operators illustrates the strength of universal inequalities: they yield solution estimates for linear parabolic PDEs that are independent of the underlying geometry, even beyond the variational ($p=2$) setting. Another motivation for this work is the numerical analysis of nonlinear parabolic PDEs. In particular, in the derivation of conditional a posteriori error estimates, the problems considered here arise as linearizations. Conditional a posteriori error estimates rely on stability estimates with explicit constants, see, e.g., \cite{Giesselmann2026, Mizuguchi2017, Plum2019}, and the quality of the error estimators is directly linked to the size of the constants. 
\subsection{Proof strategies}
\label{sec:proof strategies}

We work directly with the parabolic PDE \eqref{eq:ACP}, using neither (vector-valued) harmonic analysis of the resolvents of $A$ nor dilation theory for $(T_t)_{t\geq 0}$. This yields particularly elementary proofs of maximal regularity beyond the Hilbert space setting.
The computations are best carried out for bounded generators of exponentially stable contraction semigroups in the case $p \geq2$. Since our inequalities are universal, we may impose these qualitative assumptions without loss of generality; see Sections~\ref{sec:approximation} and~\ref{sec:qual-ass}.

Both approaches estimate the dissipation rate of the nonlinear functional
\begin{align*}
    V(v) \coloneqq \int_0^\infty \| A T_t v\|_p^p \dd t
\end{align*}
along solutions to \eqref{eq:ACP}. This reduces maximal regularity to bounds on the nonlinear form
\begin{align*}
    H(v,h) \coloneqq \int_0^\infty \big \langle |T_t v|^{p-2} \overline{T_t v}, A T_t h \big \rangle \dd t
\end{align*}
for real-valued $v, h \in \L^p$; see Section~\ref{sec:reduction to H}. In fact, $p \Re H(Av,h)$ is the Fréchet derivative $DV(v)h$, see Lemma~\ref{lem:DV}. On the diagonal $v=h$, $\L^p$-dissipativity of $-A$ controls the integrand. The challenge is to obtain bounds off the diagonal. 

In the symmetric setting, we move $A$ onto $|T_t v|^{p-2} \overline{T_t v}$ and control $H(v,h)$ using that the modulus of $|T_t v|^{p-2} \overline{T_t v}$ is a subsolution for the parabolic operator $\frac{\d}{\d t} + A$. This strategy is carried out in Section~\ref{sec:subsolutions}. It leads to the proof of Theorem~\ref{thm:symm} under the additional assumption of positivity, which an abstract embedding theorem for real contraction semigroups removes; see Section~\ref{sec:approximation}.

In the sectorially contractive setting, our strategy is inspired by the polarization formula for bilinear forms. Namely, we find a formula that expresses the (nonlinear) form $H(v,h)$ through the variation of the diagonal form $H(v,v)$ in the purely imaginary direction $\ii h$:
\begin{align*}
        p H(v,h) = \langle |v|^{p-2}\bar{v},h\rangle + (p-1) \frac{\d}{\d \eps}\Im H(v+\ii \eps h, v+\ii \eps h)\Big|_{\eps = 0},
\end{align*}
see Section~\ref{sec:Bernstein}. For $p$ an even integer, $|z|^{p-2} \overline{z}$ is a polynomial in $z, \overline{z}$ and Bernstein's inequality allows us to control the derivative of the diagonal form through the diagonal form itself, proving Theorem~\ref{thm:non-symm}.

\subsection*{Acknowledgment} We would like to thank Sebastian Bechtel and Hannes Meinlschmidt for valuable exchange on the topic.
\subsection*{AI disclosure statement} The authors used OpenAI’s ChatGPT for language editing and consistency checks during final proofreading. The mathematical results and proofs were developed without AI assistance. The authors take full responsibility for the manuscript.
\section{Notation and background}

Most of our notation is standard. We denote the Hölder conjugate of $p \in [1,\infty]$ by $q \coloneqq \nicefrac{p}{p-1}$. We use the bilinear pairing of functions on $\Omega$ defined by 
\begin{align*}
    \langle f,g \rangle \coloneqq \int_\Omega f g \dd \mu
\end{align*}
and introduce the special function 
\begin{align*}
    \phi_p: \C \to \C, \quad z \mapsto |z|^{p-2}\bar{z}.
\end{align*}
When $p$ is clear from the context, we write $\phi = \phi_p$, so that $\|f\|_p^p = \langle \phi(f), f \rangle$.

We shall use the following well-known lemma on Nemytskii operators. It is a special case of \cite[Theorem~7]{Goldberg92}, whose proof applies equally to arbitrary measure spaces in the present setting. For a direct proof of the statement below, see \cite[Lemma 1.17]{Schnaubelt25}.
\begin{lem}
\label{lem:Nemytskii}
Let $(\Omega,\mu)$ be a measure space. Let $1<\alpha \leq r < \infty$ and let $\psi \in \Cont^1(\R^2; \R)$ be such that $\psi(0) = 0$ and $|D \psi(z)| \leq C |z|^{\alpha-1}$ for some $C$ and all $z \in \R^2$. Then the Nemytskii operator
\begin{align*}
    \Psi: \L^r(\Omega;\R^2) \to \L^{r/\alpha}(\Omega; \R), \quad v \mapsto \psi(v)
\end{align*}
is continuously Fr\'echet differentiable with derivative $ D \Psi(v) h = \nabla \psi(v) \cdot h$. In particular, the functional
\begin{align*}
    \varphi: \L^\alpha(\Omega; \R^2) \to \R, \quad v \mapsto \int_{\Omega} \psi(v) \dd \mu 
\end{align*}
is continuously Fr\'echet differentiable with derivative $D \varphi(v)h = \int_\Omega \nabla \psi(v)\cdot h \dd \mu$.
\end{lem}

\begin{rem}
Identifying $\R^2 \cong \C$, the lemma also applies to real continuously differentiable functions $\psi: \C \to \R$ with the same assumptions. In this case the scalar product becomes $v \cdot h = \Re (\overline{v}h)$.
\end{rem}
\section{The Hilbert space case}
\label{sec:L2}
In this short section, we give an explicit bound for $\MR_2(A)$ that is folklore, although we could not find a reference in the literature. The point is that this bound is sharper than those provided by Theorems~\ref{thm:symm} and \ref{thm:non-symm}  when $p=2$.

\begin{prop}
\label{prop:L2}
If $-A$ generates a real, sectorially contractive semigroup on a $\sigma$-finite measure space $(\Omega,\mu)$, then 
\begin{align*}
    \MR_2(A) \leq \frac{1}{\sin \theta_2}.
\end{align*}
In particular, $\MR_2(A) \leq 1$ if $A$ is self-adjoint.
\end{prop}
\begin{proof}
    By Plancherel's Theorem, we know that 
    \begin{align*}
        \MR_2(A) \leq \sup_{s \in \R \setminus \{0\}} \|A(\ii s + A)^{-1}\|_{2 \to 2}, 
    \end{align*}
    see \cite[Corollary 17.3.8]{HvNVW3}.
    Thus, it suffices to take $f \in \L^2$ and $s \in \R \setminus \{0\}$, set $u\coloneqq (\ii s + A)^{-1}f$ and prove $\|A u\|_2 \leq (\sin \theta_2)^{-1} \|f\|_2$. 
    
    To this end, we abbreviate $z \coloneqq \langle A u , \overline{u} \rangle$. The Lumer--Phillips theorem yields $|\arg(z)| \leq \nicefrac{\pi}{2} - \theta_2$. Consequently, we have
    \begin{align}
    \label{ineq:Imz2}
        |\Im z| \leq \cos(\theta_2)|z|,
    \end{align}
    from which we obtain
    \begin{align*}
        \|Au\|_2^2 = \|f - \ii s u\|_2^2 &= \|f\|_2^2 + \|s u\|_2^2 + 2\Re \langle f, - \overline{\ii s u}  \rangle \\
        &= \|f\|_2^2 + \|s u\|_2^2 + 2\Re \langle \ii s u + Au, \ii s \bar{u}  \rangle \\
        &= \|f\|_2^2 - \|s u\|_2^2 - 2s\Im z.
    \end{align*}
    Using \eqref{ineq:Imz2} and the Cauchy--Schwarz inequality, we conclude
    \begin{align*}
        \|su\|_2^2 + \|Au\|_2^2 
        \leq \|f\|_2^2 + \|\cos(\theta_2) Au\|_2^2 + \|s u\|_2^2,
    \end{align*}
    which yields the desired bound
    \begin{align*}
        \|A u\|_2^2 \leq \frac{1}{1-\cos^2(\theta_2)}\|f\|_2^2 = \frac{1}{\sin^2(\theta_2)}\|f\|_2^2. &\qedhere
    \end{align*}
\end{proof}
\begin{rem}
    The proof above applies \emph{in extenso} to all $m$-sectorial operators $A$ on a Hilbert space; see \cite[Section 3]{AVV26} for background.
\end{rem}
\section{Four lemmas for real contraction semigroups}
\label{sec:approximation}

This section contains four lemmas for real contraction semigroups. Their sole purpose is to justify the qualitative assumptions introduced in the next section, which will facilitate the proofs of our main results. Readers can skip this section on a first reading. 

\begin{defn}
\label{def:Yosida type}
If $-A$ generates a real contraction semigroup on a $\sigma$-finite measure space $(\Omega,\mu)$, we define the Yosida-type approximants 
\begin{align*}
    A_n \coloneqq n A (n+A)^{-1}+\frac{1}{n}, \quad n \in \N.
\end{align*}
\end{defn}

The following lemma shows that the additional assumptions in our main results are preserved under the passage from $A$ to one of the bounded operators $A_n$.

\begin{lem}
\label{lem:yosida}
    Let $-A$ generate a real contraction semigroup $(T_t)_{t \geq 0}$ on a $\sigma$-finite measure space $(\Omega,
    \mu)$ and let $n \in \N$. Then $-A_n$ generates a real contraction semigroup $\smash{(T_t^{(n)})_{t \geq 0}}$ on $(\Omega,\mu)$ and for every $p \in [1,\infty]$ we have
    \begin{align*}
        \|T_t^{(n)}\|_{p \to p} \leq \e^{-t/n}, \qquad t \geq 0.
    \end{align*}
    Moreover, the following properties pass on from $(T_t)_{t \geq 0}$ to $(T_t^{(n)})_{t \geq 0}$:
    \begin{enumerate}[(a)]
        \item\label{item1:yosida} positivity,
        \item\label{item2:yosida} symmetry,
        \item\label{item3:yosida} sectorial contractivity of angle $\theta_p \leq \arccos|1 - \nicefrac{2}{p}|$.
    \end{enumerate}
\end{lem}

\begin{proof}
    By definition and the Laplace transform, we can write
    \begin{align}
        A_n 
        &= n + \frac{1}{n} - n^2(n+A)^{-1},  \label{eq:An} \\
        (n+A)^{-1} 
        &= \int_0^\infty \e^{-nt}T_t \dd t.\label{eq:resA}
    \end{align}
    This shows that $(n+A)^{-1}$ is real and that $\|(n+A)^{-1}\|_{p \to p} \leq \nicefrac{1}{n}$ holds for every $p \in [1,\infty]$. The associated semigroup is given as the operator exponential function
    \begin{align}
    \label{eq:series}
        T_t^{(n)} = \e^{-t/n} \e^{-nt} \sum_{k=0}^\infty \frac{(tn^2(n+A)^{-1})^k}{k!} .
    \end{align}
    Thus, $(T_t^{(n)})_{t \geq 0}$ is real and satisfies the claimed decay estimate. In particular, it is a real contraction semigroup on $(\Omega,\mu)$. Assertion \ref{item1:yosida} then follows from \eqref{eq:resA} and \eqref{eq:series}. Symmetry of the semigroup is equivalent to self-adjointness of the generator. Thus, \ref{item2:yosida} follows from \eqref{eq:An}. 
    
    As for \ref{item3:yosida} we let $f \in \L^p$ and set $z \coloneqq \langle A_n f, \phi(f) \rangle$. By the Lumer--Phillips theorem,
    it suffices to check that $|\arg(z)| \leq \nicefrac{\pi}{2} - \theta_p$. To this end,  we set $u\coloneqq n(n+A)^{-1}f$, so that by \eqref{eq:An} we get
    \begin{align*}
        z 
        &= \big \langle n(f-u) + \tfrac1n f, \phi(f) \big \rangle \\
        &=  \big \langle n( f-u) + \tfrac{1}{n}f, \phi(f) - \phi(u) \big \rangle + \big \langle Au + \tfrac1n f,\phi(u) \big \rangle \\
        &= n \big \langle f - u , \phi(f) - \phi(u) \big \rangle + \big \langle Au, \phi(u) \big \rangle + \tfrac1n \big \langle f , \phi(f) \big \rangle \\
        &\eqqcolon z_1 + z_2 + \tfrac1n \|f\|_p^p.
    \end{align*}
    Again by Lumer--Phillips, we have $|\arg(z_2)| \leq \nicefrac{\pi}{2} - \theta_p$. As for $z_1$, the Liskevich--Perelmuter inequality~\cite[Lemma 2]{Haase2016} yields  $|\arg(z_1)| \leq \nicefrac{\pi}{2} - \arccos| 1 - \nicefrac{2}{p} |$.
    Since $\theta_p \leq \arccos | 1- \nicefrac{2}{p}|$ by assumption,  we conclude $|\arg(z)| \leq \nicefrac{\pi}{2} - \theta_p$ as desired.
\end{proof}

A real, symmetric contraction semigroup can be embedded into a positive, symmetric contraction semigroup in the following sense. The construction has been inspired by results in \cite{Haase2014}. It uses the modulus semigroup, whose properties needed below can all be found in \cite[Section~2.4]{Ouhabaz2005}.

\begin{lem}
\label{lem:smg embedding}
Let $(T_t)_{t \geq 0}$ be a real, symmetric contraction semigroup on a $\sigma$-finite measure space $(\Omega, \mu)$. Then there is a positive, symmetric contraction semigroup $(S_t)_{t \geq 0}$ on another $\sigma$-finite measure space $(\Sigma, \lambda)$ and a linear map $j$ that is an isometry $j: \L^p(\Omega,\mu) \to \L^p(\Sigma, \lambda)$ for every $p \in [1,\infty]$ such that
\begin{align}
\label{eq:j intertwines}
    j(T_t v) = S_t j(v), \quad t \geq 0,\, v \in \L^2(\Omega, \mu).
\end{align}
\end{lem}

\begin{proof}
We consider the disjoint union $\Sigma \coloneqq \Omega \sqcup \Omega$ and identify functions $w: \Sigma \to \C$ canonically with a pair of functions $(w_1, w_2)^\top: \Omega \to \C^2$. The measure on $\Sigma$ is given by
    \begin{align*}
        |w| \dd \lambda = \frac{1}{2} |w_1| \dd \mu + \frac{1}{2} |w_2| \dd \mu.
    \end{align*}
Then $j(v) \coloneqq (v,-v)^\top$ defines an isometry of the respective $\L^p$-spaces for every $p \in [1,\infty]$. Let $(T_t^\odot)_{t \geq 0}$ be the modulus semigroup of $(T_t)_{t \geq 0}$: This is a positive, symmetric contraction semigroup on $(\Omega,\mu)$ with the property that $|T_tv| \leq T_t^\odot |v|$ for all $v \in \L^2$. On $\L^2(\Sigma, \lambda)$ we define the symmetric operators
    \begin{align*}
        S_t  
        \coloneqq \frac{1}{2} \begin{pmatrix}
            T_t^\odot + T_t& T_t^\odot - T_t \\ T_t^\odot - T_t & T_t^\odot + T_t
        \end{pmatrix}, \qquad t \geq 0,
    \end{align*}
so that \eqref{eq:j intertwines} follows by inspection. It remains to show that $(S_t)_{t \geq 0}$ is a positive contraction semigroup on $(\Sigma, \lambda)$.

The family $(S_t)_{t \geq 0}$ is strongly continuous since the two semigroups are. The semigroup property becomes obvious once we realize that 
    \begin{align*}
        S_t = U \begin{pmatrix} T_t^\odot & 0 \\ 0 & T_t \end{pmatrix} U^{-1}, \quad U \coloneqq \frac{1}{\sqrt{2}} \begin{pmatrix} 1 & 1 \\ 1 & -1 \end{pmatrix}.
    \end{align*}
Now, fix $t \geq 0$. For $0 \leq v \in \L^2(\Omega, \mu)$ we have $|T_t v| \leq T_t^\odot v$ and therefore $(T_t^\odot \pm T_t)v \geq 0$. Hence, all four entries in the matrix representation of $S_t$ are positive. If $0 \leq w \in \L^2(\Sigma, \lambda)$, then we can apply this observation to $v=w_1, w_2$ in order to conclude that $S_t w \geq 0$. Consequently, $S_t$ is positive. Similarly, for $w \in \L^2(\Sigma,\lambda) \cap \L^\infty(\Sigma,\lambda)$ we set $v \coloneqq \max(|w_1|,|w_2|)$ and obtain $|S_t w| \leq S_t(v,v)^\top$ by positivity. Since $ S_t(v,v)^\top = (T_t^\odot v, T_t^\odot v)^\top$, we conclude
    \begin{align*}
        \|S_t w\|_\infty \leq \|T_t^\odot v\|_\infty \leq \|v\|_\infty = \|w\|_\infty,
    \end{align*}
that is, the contraction property $\|S_t\|_{p \to p} \leq 1$ for $p=\infty$. Since $S_t$ is symmetric, we also get the contraction property for $p=1$ and then for all $p \in [1,\infty]$ by the Riesz--Thorin theorem.
\end{proof}
We use the preceding two lemmas to discuss suitable permanence properties of maximal regularity, that is, how the estimate \eqref{eq:MR} for some operators $A$, right-hand sides $f$ and exponents $p$ can be transferred to broader classes. For $f$ we use the following class, which is dense in $\L^p((0,\infty);\L^p)$ for every $p \in (1,\infty)$.

\begin{defn}
\label{def:tensor functions}
We denote by $\Cont^1_c(0,\infty) \otimes (\L^1 \cap \L^\infty)$ the class of continuously differentiable functions with compact support on $(0,\infty)$ and finite dimensional range in $\L^1 \cap \L^\infty$.
\end{defn}

\begin{lem}
\label{lem:permanence MR}
Let $-A$ generate a real contraction semigroup $(T_t)_{t \geq 0}$ on a $\sigma$-finite measure space $(\Omega,\mu)$. Let $p \in (1,\infty)$ and, given $f \in \L^p((0,\infty); \L^p)$, let $u$ be the mild solution to the abstract Cauchy problem \eqref{eq:ACP}. 

Let $C \geq 0$ be a constant. Each of the following properties implies that $\MR_p(A) \leq C$.
    \begin{enumerate}[(a)]
    \item\label{item1:permanence MR} $\MR_{q}(A^*) \leq C$, 
    \item\label{item2:permanence MR} \eqref{eq:MR} holds for $f \in \Cont^1_c(0,\infty) \otimes (\L^1 \cap \L^\infty)$,  
    \item\label{item3:permanence MR} \eqref{eq:MR} holds for all real-valued $f \in \Cont^1_c(0,\infty) \otimes (\L^1 \cap \L^\infty)$,
    \item\label{item4:permanence MR} $\MR_p(A_n) \leq C$ for all $n \in \N$.
    \end{enumerate}
\end{lem}

\begin{proof} 
Assertions \ref{item1:permanence MR} and \ref{item2:permanence MR} are well-known, see \cite[Proposition~17.2.32]{HvNVW3} and \cite[Theorem~17.2.19]{HvNVW3}. It remains to show that either of \ref{item3:permanence MR} and \ref{item4:permanence MR} implies \ref{item2:permanence MR}.

`\ref{item3:permanence MR} $\Longrightarrow$ \ref{item2:permanence MR}': 
    Let $f \in \Cont^1_c(0,\infty) \otimes (\L^1 \cap \L^\infty)$ and recall Zygmund's identity~\cite[Lemma 1]{MS11}: There is a constant $\gamma_p>0$ such that
    \begin{align*}
        \int_0^\infty \|g(t)\|_p^p \dd t = \gamma_p \int_0^{2 \pi}  \int_0^\infty \|\Re(\e^{\ii s}g(t))\|_p^p \dd t \dd s, \qquad g \in \L^p((0,\infty);\L^p).
    \end{align*}
    Writing $u_f$ for the mild solution of \eqref{eq:ACP} with right-hand side $f$, we have $\Re u_f = u_{\Re f}$ and $\Re (Au_f) = A u_{\Re(f)}$ since the semigroup is real. Hence, if we have \eqref{eq:MR} under the additional assumption that $f$ is real-valued, then
    \begin{align*}
        \int_0^\infty \|Au_f(t)\|_p^p \dd t 
        &= \gamma_p \int_0^{2 \pi} \int_0^\infty  \|Au_{\Re(\e^{\ii s} f)}(t)\|_p^p \dd t \dd s \\
        &\leq \gamma_p C^p \int_0^{2 \pi} \int_0^\infty \|\Re(\e^{\ii s} f(t))\|_p^p \dd t \dd s
        = C^p \int_0^\infty \|f(t)\|_p^p \dd t.
    \end{align*}
  
`\ref{item4:permanence MR} $\Longrightarrow$ \ref{item2:permanence MR}': 
    Again, let $f  \in \Cont^1_c(0,\infty) \otimes (\L^1 \cap \L^\infty)$. We fix $p \in (1,\infty)$, an arbitrary but finite $T>0$ and we write 
    \begin{align*}
        u(t) = \int_0^t T_{t-s}f(s) \dd s, \quad u_n(t) \coloneqq \int_0^t T_{t-s}^{(n)}f(s) \dd s
    \end{align*}
    for the mild solutions to the respective Cauchy problems. We pass to the limit in $n$ as follows.

    According to the Hille--Yosida theorem we have $\smash{T_t^{(n)}}v \to T_tv$ in $\Cont([0,T];\L^p)$ for every $v \in \L^p$.
    The special form of $f$ lets us conclude $u_n \to u$ in $\Cont([0,T];\L^p)$. Due to \ref{item4:permanence MR} and reflexivity we find a subsequence, also indexed by $n$ for simplicity, and a function $g \in \L^p((0,T); \L^p)$ such that $A_n u_n \to g$ weakly in $\L^p((0,T); \L^p)$. 
    
    In order to identify $g$ and $Au$, we work in $\L^p((0,T); \L^p)$. We set $w_n \coloneqq n(n+A)^{-1} u_n$ and note that $(w_n)$ is bounded due to the Hille--Yosida theorem and since the convergent sequence $(u_n)$ is bounded. Writing
    \begin{align*}
        w_n = u_n - \frac{1}{n} A w_n, \quad A w_n = A_n u _n - \frac{1}{n} u_n,
    \end{align*}
    we see that $w_n \to u$ strongly and $Aw_n \to g$ weakly. Since $A$ is a closed operator in $\L^p$, the subspace
    \begin{align*}
        W \coloneqq \biggl\{ \begin{pmatrix} w \\ A w \end{pmatrix} : w \in \L^p((0,T); \dom(A)) \biggr\} \subseteq \L^p((0,T); \L^p) \times \L^p((0,T); \L^p)
    \end{align*}
    is closed, hence weakly closed. Consequently, we have $(u,g) \in W$, which means that $u(t) \in \dom(A)$ for almost every $t \in (0,T)$ and $Au=g$. Hence, $A_n u_n \to Au$ weakly and we conclude
    \begin{align*}
        \|Au\|_{\L^p((0,T); \L^p)} \leq \liminf_{n \to \infty} \|A_n u_n\|_{\L^p((0,T); \L^p)} \leq C  \|f\|_{\L^p((0,T); \L^p)} ,
    \end{align*}
    where the final step follows from the assumption \ref{item4:permanence MR}. Since $T>0$ was arbitrary, \ref{item2:permanence MR} follows.  
\end{proof}
\begin{lem}
\label{lem:MR embedding}
Let $-A$ generate a real, symmetric contraction semigroup on a $\sigma$-finite measure space $(\Omega, \mu)$. Then there is a generator $-B$ of a positive, symmetric contraction semigroup on another $\sigma$-finite measure space $(\Sigma, \lambda)$ such that $\MR_p(A) \leq \MR_p(B)$ for every $p \in (1,\infty)$.
\end{lem}

\begin{proof}
    We write $(T_t)_{t \geq 0}$ for the semigroup generated by $-A$, let $(S_t)_{t \geq 0}$ be as in Lemma~\ref{lem:smg embedding} and call $-B$ its generator. Let $f \in \Cont_c^1(0,\infty) \otimes (\L^1(\Omega,\mu) \cap \L^\infty(\Omega,\mu))$ and let $u$ be the corresponding mild solution to the abstract Cauchy problem \eqref{eq:ACP} for $A$. We assume that $\MR_p(B) \leq C$ and need to show \eqref{eq:MR}, see Lemma~\ref{lem:permanence MR}.
    
    By \eqref{eq:j intertwines}, we see that $j(u)$ is the mild solution to the abstract Cauchy problem for $B$ with right-hand side $j(f) \in \Cont_c^1(0,\infty) \otimes (\L^1(\Sigma,\lambda) \cap \L^\infty(\Sigma,\lambda))$. Maximal regularity yields $j(u(s)) \in \dom(B)$ for almost every $s >0$ along with the estimate
    \begin{align}
    \label{eq1:MR embedding}
    \|Bj(u)\|_{\L^p((0,\infty);\L^p(\Sigma,\lambda))}  \leq C \|j(f)\|_{\L^p((0,\infty);\L^p(\Sigma,\lambda))}.
    \end{align} 
    Differentiating \eqref{eq:j intertwines} at $t=0$, it follows that $u(s) \in \dom(A)$ with $j(Au(s)) = B j(u(s))$ for the same $s>0$. Since $j$ is an isometry between the respective $\L^p$-spaces, \eqref{eq1:MR embedding} implies \eqref{eq:MR}.
\end{proof}
\section{The qualitative assumptions}
\label{sec:qual-ass}

Summarizing the lemmas from the previous section and since we have already treated the case $p = 2$ in Section~\ref{sec:L2}, it suffices to prove \eqref{eq:MR} for $p>2$ and real-valued $f \in \Cont_c^1(0,\infty) \otimes (\L^1 \cap \L^\infty)$ under the following additional qualitative assumptions:
\begin{enumerate}[(i)]
    \item $A$ is bounded on $\L^r$ for every $r \in [1,\infty]$ ,
    \item there is $c>0$ such that $\|T_t\|_{r \to r} \leq \e^{-ct}$ for every $t \geq 0$ and every $r \in [1,\infty]$.
\end{enumerate} 
For Theorem~\ref{thm:symm}, Lemma~\ref{lem:MR embedding} also allows us to strengthen the assumption of realness to positivity without loss of generality. In this scenario, the mild solution $u: [0,\infty) \to \L^p$ is continuously differentiable, real-valued and solves \eqref{eq:ACP} in the strong sense. Moreover, we have
\begin{align*}
    \|A u(t)\|_p \leq \|A\|_{p \to p} \int_0^\infty \one_{[0,\infty)}(t-s) \e^{-c (t-s)} \|f(s)\|_p \dd s, \quad t \geq 0.
\end{align*}
Thus, Young's convolution inequality yields $Au \in \L^p((0,\infty);\L^p)$ qualitatively.

\section{Reduction to a nonlinear form bound}
\label{sec:reduction to H}

We now carry out the first step of the strategy outlined in Section~\ref{sec:proof strategies}. We maintain the qualitative assumptions from Section~\ref{sec:qual-ass}. In particular, $p > 2 $ is fixed and the following two quantities are well-defined.
\begin{defn}
\label{defn: V and H} 
For $v,h \in \L^1 \cap \L^\infty$ define
\begin{align*}
    V(v) \coloneqq \int_0^\infty \| A T_t v\|_p^p \dd t 
    \quad \text{and} \quad H(v,h) \coloneqq \int_0^\infty \langle \phi(T_t v), A T_t h \rangle \dd t.
\end{align*}
\end{defn}
\begin{lem}
\label{lem:DV}
Let $v,h \in \L^1 \cap \L^\infty$. Then $V$ is (real) Fr\'echet differentiable at $v$ and we have
\begin{enumerate}[(a)]
    \item\label{item:DVvh} $D V(v) h = p \Re H(Av, h)$,
    \item \label{item:DVvAv} $DV(v) Av = \|Av\|_p^p$.
\end{enumerate}
\end{lem}

\begin{proof}
The functional $V$ is the composition of the bounded linear orbit map
    \begin{align*}
        \L^p(\Omega) \to \L^p((0,\infty) \times \Omega), \quad v \mapsto (AT_tv)_{t > 0}
    \end{align*}
and the $p$-th power of the norm on $\L^p((0,\infty) \times \Omega)$, which is (real) Fr\'echet differentiable by Lemma~\ref{lem:Nemytskii}. The chain rule yields that $V$ is (real) Fr\'echet differentiable with derivative
    \begin{align*}
        D V(v) h = p \int_0^\infty \Re \langle \phi(A T_t v), A T_t h \rangle \dd t = p \Re H(Av,h).
    \end{align*}
This proves \ref{item:DVvh}. Next, we note that for $\tau \geq 0$ the semigroup law yields
    \begin{align*}
        V(T_\tau v) = \int_\tau^\infty \|A T_t v\|_p^p \dd t.
    \end{align*} 
Differentiating either side of the expression with respect to $\tau$, we find
    \begin{align*}
      D V(T_\tau v)(-A T_\tau v) =  \frac{\d}{\d \tau} V(T_\tau v) =  - \|A T_\tau v\|_p^p
    \end{align*}
and evaluating at $\tau = 0$, we obtain \ref{item:DVvAv}.
\end{proof}
We can now reduce the problem of bounding $\MR_p(A)$ to finding upper bounds for the nonlinear form $H$.
\begin{prop}
\label{prop:MR}
Let $\kappa \geq 0$. Then $\MR_p(A) \leq p \kappa$ provided that for all real-valued $v,h \in \L^1 \cap \L^\infty$ with $\|v\|_p = \|h\|_p = 1$ we have
    \begin{align}
        \label{ineq:Hkappa}
        \Re H(v,h) \leq \kappa.
    \end{align}
\end{prop}
\begin{proof}
Let $u$ be the mild solution to the abstract Cauchy problem \eqref{eq:ACP} with right-hand side $f$ as specified in the qualitative assumptions. From Section~\ref{sec:qual-ass} we know that $u: [0,\infty) \to \L^p$ is continuously differentiable. Consequently, we can compute the dissipation rate of $V$ along $u$ by the chain rule, taking into account Lemma~\ref{lem:DV}:
    \begin{align*}
        \frac{\d}{\d t} V(u(t)) 
        = D V(u(t))u^\prime(t) 
        = p \Re H(A u(t), f(t)) - \|A u(t)\|_p^p.
    \end{align*}
Integrating both sides with respect to $t \in [0,T]$ for some finite $T>0$ and using $V(0) = 0$ yields the energy estimate
    \begin{align*}
        V(u(T)) = \int_0^T p \Re H(A u(t),f(t)) - \|A u(t)\|_p^p \dd t.
    \end{align*}
Since $V(u(T)) \geq 0$, we can re-arrange terms and take the supremum in $T$ to obtain
    \begin{align*}
        \int_0^\infty \|A u(t) \|_p^p \dd t \leq p \int_0^\infty H(A u(t),f(t)) \dd t.
    \end{align*}
By homogeneity and since $H$ vanishes if either of its variables is zero, \eqref{ineq:Hkappa} implies that for all real $v,h \in \L^1 \cap \L^\infty$ we have $\Re H(v,h) \leq \kappa \|v\|_p^{p-1} \|h\|_p$. The qualitative assumption lets us use this bound with $v = Au(t)$, $h = f(t)$ for every $t>0$. Thus,
\begin{align*}
     \int_0^\infty \|A u(t) \|_p^p \dd t 
     &\leq p \kappa  \int_0^\infty \|A u(t)\|_p^{p-1} \|f(t)\|_p \dd t \\
     &\leq p \kappa  \biggl( \int_0^\infty \|A u(t)\|_p^p \dd t\biggr)^{1-1/p}  \biggl( \int_0^\infty  \|f(t)\|_p^p \dd t \biggr)^{1/p},
\end{align*}
where we used Hölder's inequality in the second step. Since the integral on the left is finite by the qualitative assumptions, we obtain \eqref{eq:MR} with $C = p \kappa$, as claimed.
\end{proof}
\section{Proof of Theorem~\ref{thm:non-symm}}
\label{sec:Bernstein}
In this section we prove Theorem~\ref{thm:non-symm} under the qualitative assumptions. First, we illustrate how contractivity of the semigroup yields bounds for $H$ from Definition~\ref{defn: V and H} on the diagonal
\begin{align*}
    H(v) \coloneqq H(v,v)
\end{align*}
through the Lumer--Phillips theorem.
\begin{lem}
\label{lem:ReImH}
If $v \in \L^1 \cap \L^\infty$, then
    \begin{enumerate}[(a)]
        \item\label{item:ReH} $\Re H(v) = \frac{1}{p} \|v\|_p^p$,
        \item\label{item:ImH} $|\Im H(v)| \leq \frac{\cot(\theta_p)}{p} \|v\|_p^p$.
    \end{enumerate}
\end{lem}

\begin{proof}
    As in the proof of Lemma~\ref{lem:DV}, (real) Fr\'echet differentiability of the $p$-th power of the norm on $\L^p$ yields
    \begin{align*}
        - \frac{\d}{\d t} \|T_t v\|_p^p = p \Re \langle \phi(T_t v) , A T_t v \rangle .
    \end{align*}
    Integrating both sides in $t\geq 0$ and using $\|T_t v\|_p \to 0$ as $t \to \infty$ by the qualitative assumption gives \ref{item:ReH}. For \ref{item:ImH} we set $z(t) \coloneqq \langle \phi(T_t v),A T_t v \rangle$. By sectorial contractivity and the Lumer--Phillips theorem we get $|\arg z(t)| \leq \nicefrac{\pi}{2} - \theta_p$ and we conclude 
    \begin{align*}
        |\Im z(t)| \leq \cot(\theta_p) \Re z(t).
    \end{align*}
    Now, \ref{item:ImH} follows upon integrating both sides in $t\geq 0$ and using \ref{item:ReH}.
\end{proof}
For the rest of the section, we retain the additional assumption of Theorem~\ref{thm:non-symm} that $p =2n$ is an even integer. According to Proposition~\ref{prop:MR} we need to show
\begin{align}
\label{eq:goal section 4}
\Re H(v,h) \leq \frac{1}{p} + \frac{p-1}{p} \cot(\theta_p)
\end{align}
for all real-valued $v,h \in \L^1 \cap \L^\infty$ that are normalized to $\|v\|_p = \|h\|_p = 1$. Henceforth, such $v,h$ are fixed. 
We will bound $H(v,h)$ by studying the variation of $H(v)$ in directions $h$ and $\ii h$. 

Since $p$ is an even integer, $\phi(z) = |z|^{p-2}\bar{z} = (z \bar{z})^{n-1} \bar{z}$ is a polynomial in $z$ and $\bar{z}$ and for real numbers $\alpha, \beta$ we find
\begin{align*}
    H(\alpha v + \beta h) 
    &= \int_0^\infty \Big \langle \bigl(\alpha T_t v + \beta T_t h\bigr)^{p-1}, \alpha A T_t v + \beta A T_t h\Big \rangle \dd t, \\
    H(\alpha v + \ii \beta h) &= \int_0^\infty \Big \langle \bigl(\alpha^2 (T_t v)^2 + \beta^2 (T_t h)^2\bigr)^{n-1}\bigl(\alpha T_t v - \ii \beta T_t h\bigr), \alpha A T_t v + \ii \beta A T_t h\Big \rangle \dd t.
\end{align*}
Expansion shows that both are polynomials in $\alpha$ and $\beta$.
Differentiation at $\beta = 0$ amounts to identifying the linear terms in $\beta$.
For $\alpha = 1$ and $\beta = \eps$ we get
\begin{align}
    \frac{\d}{\d \eps} \Re H(v+\eps h)\Big|_{\eps = 0} 
    & = H(v,h) + (p-1) \int_0^\infty \big \langle (T_t v)^{p-2} T_t h, A T_t v  \big \rangle \dd t, \label{eq:ddepsRe} \\
    \frac{\d}{\d \eps} \Im H(v+\ii \eps h)\Big|_{\eps = 0} 
    &= H(v,h) - \int_0^\infty \big \langle (T_t v)^{p-2} T_t h, A T_t v  \big \rangle \dd t   \label{eq:ddepsIm}.
\end{align}
The derivative in \eqref{eq:ddepsRe} can also be computed using Lemma~\ref{lem:ReImH}~\ref{item:ReH} and differentiating the $p$-th power of the norm on $\L^p$:
\begin{align*}
    \frac{\d}{\d \eps} \Re H(v+\eps h)\Big|_{\eps = 0}  = \langle\phi(v),h\rangle.
\end{align*}
Adding $(p-1)$ times \eqref{eq:ddepsIm} to \eqref{eq:ddepsRe} eliminates the integral term and leaves us with
\begin{align}
\label{eq:pH}
    p H(v,h) = \langle \phi(v),h\rangle + (p-1) \frac{\d}{\d \eps}\Im H(v+\ii \eps h)\Big|_{\eps = 0}.
\end{align}

The obstacle is now clear: Lemma~\ref{lem:ReImH} shows that contractivity of the semigroup bounds the magnitude of $H$, but does not control its directional derivative in the imaginary direction. Bernstein's inequality~\cite{Riesz1914} for trigonometric polynomials provides a very convenient way out. Indeed,  
\begin{align*}
    P(\theta) \coloneqq \Im H(\cos(\theta)v + \ii \sin(\theta) h)
\end{align*}
is a trigonometric polynomial of degree at most $2n = p$ and expanding $P(\theta)$ and using $\cos^\prime(0) = 0$ and $\sin^\prime(0) = 1$, we find that the derivative in \eqref{eq:pH} can be re-interpreted as
\begin{align*}
    P^\prime(0) = \frac{\d}{\d \eps} \Im H(v+\ii \eps h)\Big|_{\eps = 0}.
\end{align*} 
Lemma~\ref{lem:ReImH} controls the magnitude of $P$ and we find for every $\theta \in \R$ that
\begin{align*}
    |P(\theta)| 
    &\leq \frac{\cot(\theta_p)}{p} \| \cos(\theta)v + \ii \sin(\theta)h\|_p^p \\
    &= \frac{\cot(\theta_p)}{p} \| \cos^2(\theta)v^2 + \sin^2(\theta)h^2\|_n^n \\
    &\leq \frac{\cot(\theta_p)}{p} \bigl( \cos^2(\theta) \|v^2\|_n + \sin^2(\theta) \|h^2\|_n\bigr)^n \\
    &= \frac{\cot(\theta_p)}{p}.
\end{align*}
The upshot is that Bernstein's inequality yields
\begin{align*}
    |P'(0)| \leq p \sup_{\theta \in \R} |P(\theta)| \leq \cot(\theta_p).
\end{align*}

Going back to \eqref{eq:pH}, we conclude
\begin{align*}
    p|H(v,h)| 
    \leq |\langle \phi(v),h\rangle| + (p-1) \cot(\theta_p)
    \leq 1 + (p-1) \cot(\theta_p),
\end{align*}
where the second step follows from Hölder's inequality and the normalization of $v$ and $h$. Dividing by $p$ gives \eqref{eq:goal section 4} and the proof is complete.\hfill\qed{\parfillskip0pt\par}
\section{Proof of Theorem~\ref{thm:symm}}
\label{sec:subsolutions}

In this section we prove Theorem~\ref{thm:symm} under the qualitative assumptions. In particular, we also assume that $(T_t)_{t \geq 0}$ is positive, which, in turn, will allow us to estimate nonlinear terms through the following inequalities for submarkovian semigroups.
\begin{prop}[{\cite[Proposition~3.4, Theorem~3.5]{Ouhabaz2005}}]
    \label{prop:Ouhabaz}
     If $j \in \Cont^1(\R; \R)$ is convex and satisfies $j(0) = 0$, and $w\in \L^1 \cap \L^\infty$ is real-valued, then the following hold true.
     \begin{enumerate}[(a)]
         \item Jensen's inequality: $j(T_\eps w) \leq T_\eps j(w)$ for all $ \eps \geq 0$. \label{Jensen}
         \item Kato's inequality: $A j(w) \leq j'(w) Aw$, provided that $j'$ is bounded. \label{Kato}
    \end{enumerate}
\end{prop}
As in the previous section, we invoke Proposition~\ref{prop:MR}. This time it reduces the claim to establishing the bound 
\begin{align}
\label{eq:goal section 5}
\Re H(v,h) \leq \frac{1+q}{2}
\end{align}
for all real-valued functions $v,h \in \L^1 \cap \L^\infty$ normalized to $\|v\|_p = \|h\|_p = 1$. Indeed,
\begin{align*}
    \frac{p(1+q)}{2} = \frac{p(2p-1)}{2(p-1)} = p + \frac{p}{2(p-1)}
\end{align*}
is the bound stated in Theorem~\ref{thm:symm}. For the rest of the section such $v,h$ are fixed. 

In this proof we use the parabolic operator $\frac{\d}{\d t} + A$ more explicitly. Concerning derivatives of the nonlinear terms we note that the functions $t\mapsto \phi(T_t v)$ and $t \mapsto |\phi(T_t v)|$ are of class $\Cont^1([0,\infty); \L^q)$. Indeed, since $p>2$, this follows directly from the qualitative assumptions, Lemma~\ref{lem:Nemytskii} and the chain rule.

We start the proof of \eqref{eq:goal section 5} by observing that
     \begin{align*}
         - \frac{\d}{\d t} \langle \phi(T_t v), T_t h \rangle = \langle \phi(T_t v), A T_t h \rangle - \Big \langle \frac{\d}{\d t} \phi(T_t v), T_t h \Big \rangle.
     \end{align*}
Since $\|T_th\|_p \to 0$ as $t \to \infty$ by the qualitative assumption, integration in $t\geq 0$ yields
     \begin{align*}
        \langle \phi(v),h \rangle 
        &= H(v,h) - \int_0^\infty \Big \langle \frac{\d}{\d t} \phi(T_t v), T_t h \Big \rangle \dd t \\
        &= H(v,h) - \int_0^\infty \Big \langle \Bigl(\frac{\d}{\d t} + A\Big) \phi(T_t v), T_t h\Big \rangle \dd t + \int_0^\infty \Big \langle A \phi(T_t v), T_t h \Big \rangle \dd t.
    \end{align*}
The upshot is that the final term is the same as $H(v,h)$ since $A$ is symmetric. Thus,
    \begin{align*}
        2 H(v,h) = \langle \phi(v),h \rangle + \int_0^\infty \Big \langle \Bigl(\frac{\d}{\d t} + A\Big) \phi(T_t v), T_t h \Big \rangle \dd t
    \end{align*}
and, using Hölder's inequality and the normalization of $v$ and $h$, this gives
    \begin{align}
    \label{eq:first intermediate result section 5}
        2 |H(v,h)| \leq 1 + \int_0^\infty \Big \langle \Big|\Bigl(\frac{\d}{\d t} + A\Big) \phi(T_t v)\Big|, |T_t h|\Big \rangle \dd t.
    \end{align}
Note that we have put absolute values on both functions in the bilinear pairing on the right. To continue, we use that $|\phi(T_t v)|$ is a subsolution to the parabolic equation. More precisely:
\begin{lem}
\label{lem:subsol}
We have
        \begin{align*}
            \Big|\Bigl(\frac{\d}{\d t} + A\Big) \phi(T_t v)\Big| \leq - \Bigl(\frac{\d}{\d t} + A \Big) | \phi(T_t v) | \quad \text{$\mu$-a.e.\ on $\Omega$.}
        \end{align*}
\end{lem}

\begin{proof}
For $x\in \R$, we have $|\phi(x)| \pm \phi(x) =  2 | x |^{p-1}\one_{\R_{\pm}}(x)$, which is $\Cont^1$, convex and vanishes at $x=0$. Thus, Jensen's inequality yields for $t, \eps>0$ the pointwise estimate
    \begin{align*}
            |\phi(T_{t+\eps} v) | \pm \phi(T_{t+\eps} v) \leq T_\eps |\phi(T_t v) | \pm T_\eps \phi(T_t v),
    \end{align*}
which we write as
    \begin{align}
    \label{ineq:Jensen}
            |\phi(T_{t+\eps} v) - T_\eps \phi(T_t v) |  \leq T_\eps |\phi(T_t v) | - |\phi(T_{t+\eps} v)|.
    \end{align}
Above, we have seen that $t\mapsto \phi(T_t v)$ is of class $\Cont^1([0,\infty); \L^q)$ and the same is true for $t \mapsto T_t \phi(v)$. Hence, as a limit in $\L^q$, we can compute
    \begin{align*}
        \Bigl(\frac{\d}{\d t} + A\Big) \phi(T_t v)
        = \lim_{\eps \searrow 0}  \frac{\phi(T_{t+\eps}v) - \phi(T_t v)}{\eps} - \frac{T_\eps \phi(T_t v) - \phi(T_t v)}{\eps} 
        = \lim_{\eps \searrow 0} \frac{\phi(T_\eps T_t v) - T_\eps \phi(T_t v)}{\eps}.
    \end{align*}
In the same manner,
    \begin{align*}
        -\Bigl(\frac{\d}{\d t} + A\Big) |\phi(T_t v)|
        = \lim_{\eps \searrow 0} \frac{T_\eps |\phi(T_t v)|- |\phi(T_\eps T_t v)|}{\eps}.
    \end{align*}
Since $\L^q$-convergence preserves almost everywhere inequalities, the claim follows from \eqref{ineq:Jensen}.
\end{proof}
Applying the lemma on the right-hand side of \eqref{eq:first intermediate result section 5}, we find
    \begin{align*}
        2 |H(v,h)| 
        &\leq 1 + \int_0^\infty \Big \langle - \Bigl( \frac{\d}{\d t} + A \Big)|\phi(T_t v)|,|T_t h|\Big \rangle \dd t \\
        &\leq 1 - \int_0^\infty \Big \langle T_t \Bigl(\frac{\d}{\d t} + A \Big)|\phi(T_t v)|,|h|\Big \rangle \dd t,
    \end{align*}
where the second step uses $|T_t h | \leq T_t |h|$ as a consequence of positivity and then symmetry of $T_t$. Having stripped off all $t$-dependence from $h$, we can use $\L^q$-valued integrals and then H\"older's inequality to write 
\begin{align*}
     2 |H(v,h)| 
     &\leq 1 - \bigg \langle \int_0^\infty T_t \Bigl(\frac{\d}{\d t} + A \Big)|\phi(T_t v)| \dd t, |h| \bigg \rangle \\
     &\leq 1 + \biggl\|\int_0^\infty T_t \Bigl( \frac{\d}{\d t} + A \Big)|\phi(T_t v)| \dd t \biggr\|_q.
\end{align*}
In particular,
\begin{align}
\label{eq:second intermediate result section 5}
    2H(v,h) \leq 1 + \sup_{T>0}\; \bigg\|\int_0^T T_t \Bigl( \frac{\d}{\d t} + A \Big)|\phi(T_t v)| \dd t \biggr\|_q.
\end{align}
To assess how far we have progressed toward the goal in \eqref{eq:goal section 5}, it is instructive to interpret the right-hand side of \eqref{eq:second intermediate result section 5} as follows. Here, $T>0$ is fixed. For $t \in [0,T]$ we introduce three $\L^q$-valued functions:
\begin{align}
\label{eq:u,f,w}
    u(t)\coloneqq |\phi(T_t v)|, \quad 
    -f(t)\coloneqq u^\prime(t) + A u(t), \quad 
     w(t)\coloneqq \int_t^T T_{s-t}f(s) \dd s.
\end{align} 
Then $u$ is continuously differentiable and solves $u^\prime + Au = -f$, where the right-hand side is negative by Lemma~\ref{lem:subsol}. Our qualitative assumptions guarantee that $w$ is a continuously differentiable solution to the backwards equation $w^\prime - A w = -f$ with $w(T) = 0$. The term that we have on the right-hand side of \eqref{eq:second intermediate result section 5} is $\|w(0)\|_q$ and what we need on the right-hand side of \eqref{eq:goal section 5} is $q\|u(0)\|_q = q\|v\|_p^{p-1} = q$. Thus, we need control of $\|w(0)\|_q$ by $q\|u(0)\|_q$. 

To this end, we need two lemmas: one gathers additional qualitative properties of $u$ and $w$; the other one is a more general subsolution lemma for functions with precisely these additional properties that gives the desired comparison.
\begin{lem}
\label{lem:qualitative propertie of $u$ and $w$}
Fix $T>0$. The functions $u, w: [0,T] \to \L^q$ defined in \eqref{eq:u,f,w} take their values in the positive measurable functions on $\Omega$ and for every $r \in [1,\infty]$ there is a constant $C$ such that 
\begin{align*}
    \|u(t)\|_r + \|w(t)\|_r \leq C, \quad t \in [0,T].
\end{align*}
\end{lem}
\begin{lem}
\label{lem:backeq}
Let $p>2$ and $T> 0$. Let $u,w : [0,T] \to \L^q$ be continuously differentiable with
\begin{align*}
    u^\prime + A u = w^\prime - A w \leq 0, \quad w(T) = 0.
\end{align*}
Moreover, suppose that
\begin{enumerate}[(a)]
    \item $u$ and $w$ take their values in the positive measurable functions on $\Omega$,
    \item For every $r \in [1,\infty]$ there is a constant $C$ such that 
        \begin{align*}
            \|u(t)\|_r + \|w(t)\|_r \leq C, \quad t \in [0,T].
        \end{align*}
\end{enumerate}
Then, $\|w(0)\|_q \leq q \|u(0)\|_q$.
\end{lem}
Taking the lemmas for granted, we obtain from \eqref{eq:second intermediate result section 5} the bound $2H(v,h) \leq 1 + q$, which is \eqref{eq:goal section 5}. This completes the proof of Theorem~\ref{thm:symm} modulo the proofs of the two lemmas.
\begin{proof}[Proof of Lemma~\ref{lem:qualitative propertie of $u$ and $w$}]
Clearly $u$ is positive. By Lemma~\ref{lem:subsol} we see that $f$ is positive and since the semigroup is positive, also $w$ is positive. Now, let $r \in [1,\infty]$. Since $v \in \L^1 \cap \L^\infty$ we have the pointwise bound
\begin{align*}
    |u(t)| = |T_tv|^{p-1} \leq \|T_t v\|_\infty^{p-2} |T_tv| \leq \|v\|_\infty^{p-2} |T_tv|
\end{align*}
and therefore $\|u(t)\|_r \leq \|v\|_\infty^{p-2} \|v\|_r$. Since $A$ is bounded on $\L^r$ for every $r$ and $|u^\prime(t)| = (p-1)|T_t v|^{p-2} |A T_tv|$ by Lemma~\ref{lem:Nemytskii}, we also find a constant $C$ independent of $t \in [0,T]$ such that $ \|u^\prime(t)\|_r + \|A u(t)\|_r \leq C$. We conclude $\|f(t)\|_r \leq C$ and $\|w(t)\|_r \leq CT$.
\end{proof}
\begin{proof}[Proof of Lemma~\ref{lem:backeq}]
We set $-f \coloneqq u^\prime + A u \leq 0$. Fr\'echet differentiability of the $q$-th power of the $\L^q$-norm and the chain rule yield
    \begin{align*}
        \frac{\d}{\d t}\frac{1}{q}\|w(t)\|_q^q 
        &= \big \langle |w(t)|^{q-1},w^\prime(t) \big \rangle \\
        &= \big \langle |w(t)|^{q-1},A w(t) - f(t)\big \rangle \\
        &\geq - \big \langle |w(t)|^{q-1},f(t) \big \rangle,
    \end{align*}
where the last inequality follows from the Lumer--Phillips theorem. Since $w(T) = 0$, integration in $t$ yields
    \begin{align}
    \label{ineq:w0}
        \frac{1}{q} \|w(0)\|_q^q \leq \int_0^T \big \langle  |w(t)|^{q-1},f(t)\big \rangle \dd t.
    \end{align}
    
In order to estimate the right-hand side, we approximate bilinear pairings involving $|w|^{q-1}$. For every $\eps>0$ we pick a function $\psi_\eps \in \Cont_c^1(\R^2; \R)$ that vanishes at the origin and satisfies $D\psi_\varepsilon(z)=O(|z|)$ as $z \to 0$ such that
\begin{align}
\label{eq:goal Nemitskii shit}
    \big \langle  u, |w + \eps|^{q-1} - \eps^{q-1} \big \rangle = \int_\Omega \psi_\eps(u,w) \dd \mu,
\end{align}
where, as in the following computations, we suppress the dependence on $t$ for readability. In particular, $\psi_\eps$ fits into the setting of Lemma~\ref{lem:Nemytskii} with $\alpha = q$.

The precise construction is as follows. By assumption there is a constant $C$ such that $\|u(t)\|_\infty + \|w(t)\|_\infty \leq C$ for all $t \in [0,T]$. We pick a concave, increasing function $\rho_\eps \in \Cont^1(\R)$ with bounded derivative that coincides with $|y + \eps|^{q-1} - \eps^{q-1}$ for $0 \leq y \leq C$. This is possible since $q-1 \in (0,1)$ and $y \mapsto |y + \eps|^{q-1} - \eps^{q-1}$ is concave for $y > -\eps$. We also take $\sigma, \varrho \in \Cont_c^1(\R)$ such that $\sigma(x) = x$ and $\varrho(y) = 1$ for $x,y \in [-C,C]$. By construction and since $w$ is positive, the respective Nemytskii operators act as
\begin{align}
\label{eq1:backeq}
    \rho_\eps(w) = |w + \eps|^{q-1} - \eps^{q-1}, \quad \sigma(u) = u, \quad \varrho(w) = 1
\end{align}
and with $\psi_\eps(x,y) \coloneqq \sigma(x) \rho_\eps(y) \varrho(y)$ we indeed have \eqref{eq:goal Nemitskii shit}.
As $q-1 \in (0,1)$, we also see that 
\begin{align}
\label{eq2:backeq}
    \rho_\eps (w) \nearrow |w|^{q-1} \quad \text{as} \quad \eps \searrow 0.
\end{align}

By Lemma~\ref{lem:Nemytskii} and the chain rule, we can now differentiate both sides of \eqref{eq:goal Nemitskii shit} with respect to $t$ and obtain
    \begin{align*}
        \frac{\d}{\d t} \big \langle u , \rho_\eps(w) \big\rangle 
        &= \big \langle u^\prime, \sigma'(u) \rho_\eps(w) \varrho(w) \big \rangle + \big \langle w^\prime , \sigma(u) (\rho_\eps \varrho)^\prime(w) \big \rangle \\
        &= \big \langle u^\prime, \rho_\eps(w) \big \rangle + \big \langle u , \rho_\eps^\prime(w)w^\prime \big \rangle \\
        &= - \big \langle Au, \rho_\eps(w) \big \rangle - \big \langle f , \rho_\eps(w) \big \rangle + \big \langle u, \rho_\eps^\prime(w) A w \big \rangle - \big \langle u, \rho_\eps^\prime(w) f \big \rangle \\
        &=\big\langle u , \rho_\eps^\prime(w)Aw - A \rho_\eps(w)\big \rangle - \big \langle f, \rho_\eps(w) \big \rangle - \big \langle u,\rho_\eps^\prime(w)f\big \rangle.
    \end{align*}
On the right-hand side the third bilinear pairing is positive since $u, f$ are positive and $\rho_\eps$ is increasing. The first pairing is negative since $u$ is positive and by Kato's inequality with $j = -\rho_\eps$ we have $- A \rho_\eps(w) \leq - \rho_\eps^\prime(w) A w$, see Proposition~\ref{prop:Ouhabaz}~\ref{Kato}. Thus, we have
    \begin{align*}
        \frac{\d}{\d t}\big \langle u , \rho_\eps(w) \big \rangle \leq - \big \langle f, \rho_\eps(w) \big \rangle.
    \end{align*}
Integration in $t$, taking into account $\rho_\eps(w(T)) = \rho_\eps(0) = 0$, leads us to
    \begin{align*}
           \int_0^T \big \langle f(t),\rho_\eps(w(t))\big \rangle \dd t \leq  \big \langle u(0), \rho_\eps(w(0)) \big \rangle.
    \end{align*}
Since $f$ and $u$ are positive, we can use \eqref{eq2:backeq} and the monotone convergence theorem to pass to the limit as $\eps \searrow 0$ and conclude
    \begin{align*}
    \int_0^T \big \langle f(t),|w(t)|^{q-1} \big \rangle \dd t \leq \big \langle u(0), |w(0)|^{q-1} \big \rangle
    \leq \|u(0)\|_q \|w(0)\|_q^{q-1},
    \end{align*}
where the second step is due to H\"older's inequality. Together with \eqref{ineq:w0}, the claim follows.
\end{proof}
\section{Proofs of Corollaries~\ref{cor:Riesz-Thorin non-symmetric} and \ref{cor:Riesz-Thorin non-symmetric div}}
\label{sec:interpolation}

This section contains the proofs of the corollaries stated in the introduction, both of which follow from Theorem~\ref{thm:non-symm} by the Riesz–Thorin interpolation theorem. By duality (Lemma~\ref{lem:permanence MR}~\ref{item1:permanence MR}) it suffices to consider exponents $p \geq 2$.
We begin with the general extension of Theorem~\ref{thm:non-symm} to non-integer exponents $p \geq 2$. It relies on the observation that Stein interpolation allows us to choose $\theta_p \geq \nicefrac{2 \theta_2}{p}$, which yields
\begin{align}
\label{Steiniger Ungleichung}
   \cot(\theta_p) \leq \frac{1}{\theta_p} \leq \frac{p}{2 \theta_2},
\end{align}
see, for example, \cite[Theorem~10.10]{AVV26}.
\begin{proof}[Proof of Corollary~\ref{cor:Riesz-Thorin non-symmetric}]
For integers $n$, Theorem~\ref{thm:non-symm} together with \eqref{Steiniger Ungleichung} yields the universal bound 
    \begin{align}
    \label{eq1:Riesz-Thorin non-symmetric}
        \MR_{2n}(A) 
        &\leq 1+ \frac{n(2n-1)}{\theta_2}.
    \end{align}
Now, fix $p \geq 2$. We pick the integer $n$ with $2n \leq p \leq 2n+2$ and take $\lambda \in [0,1]$ such that $\nicefrac{1}{p} = \nicefrac{(1-\lambda)}{2n} + \nicefrac{\lambda}{(2n+2)}$. A technical subtlety forces us to split the argument into two cases.

\emph{Case 1: $n \geq 2$.} 
The Riesz--Thorin theorem yields
\begin{align*}
     \MR_p(A) 
        &\leq \MR_{2n}(A)^{1-\lambda} \MR_{2n+2}(A)^\lambda \\
        &\leq \Bigl(1+ \frac{n(2n-1)}{\theta_2}\Bigr)^{1-\lambda} \Bigl(1+ \frac{(n+1)(2n+1)}{\theta_2}\Bigr)^{\lambda}.
\end{align*}
By logarithmic concavity of $x \mapsto 1+\nicefrac{x^2}{\theta_2}$ for $x \geq \sqrt{\theta_2}$ (first line) and H\"older's inequality with conjugate exponents $2$ and $2$ (third line), we find
    \begin{align}
         \MR_p(A) 
         &\leq 1+ \frac{1}{\theta_2}\Bigl( (1-\lambda) \sqrt{n} \sqrt{2n-1} + \lambda \sqrt{n+1} \sqrt{2n+1} \Bigr)^2 \notag \\
         &= 1+ \frac{1}{\theta_2}\Bigl( \sqrt{(1-\lambda)n}\cdot \sqrt{(1-\lambda)(2n-1)} + \sqrt{\lambda(n+1)} \cdot \sqrt{\lambda(2n+1)} \Bigr)^2 \notag \\
         &\leq 1+ \frac{1}{\theta_2}\bigl(n+\lambda\bigr)\bigl(2n+2\lambda-1\bigr). \label{eq2:Riesz-Thorin non-symmetric}
\intertext{We re-introduce $p$ on the right-hand side through $p (n+1-\lambda) = 2n(n+1)$ and then apply Cauchy's inequality $ab \leq (\nicefrac{a}{2} + \nicefrac{b}{2})^2$ in order to conclude} 
         &= 1 + \frac{1}{2\theta_2} \Bigl(\frac{p}{n(n+1)}\Big)^2 \bigl(n+\lambda\bigr)\bigl(n+1-\lambda\bigr)\bigl(n+\lambda - \tfrac12\bigr) \bigl(n+1-\lambda\bigr) \notag \\ 
         &\leq 1 + \frac{1}{2\theta_2} \Bigl(\frac{p}{n(n+1)}\Big)^2 \bigl(n+\tfrac12\bigr)^2\bigl(n+\tfrac14\bigr)^2 \notag \\ 
         &= 1 + \frac{p^2}{2\theta_2} \biggl(\frac{(n+\frac12)(n+\frac14)}{n(n+1)}\biggr)^2. \notag
    \end{align}
    Comparison with the claim shows that it suffices to bound the remaining term in $n$ by $1$, which we can do by simply expanding the numerator and denominator:
    \begin{align*}
        n(n+1) - \Bigl(n+\frac12\Bigr)\Bigl(n+\frac14\Bigr) = \frac{n}{4} - \frac18  > 0.
    \end{align*}

\emph{Case 2: $n = 1$.} 
By Proposition~\ref{prop:L2} we have
    \begin{align*}
        \MR_2(A) \leq \frac{1}{\sin(\theta_2)} = \frac12 \Bigl( \tan\Bigl(\frac{\theta_2}{2}\Big) + \cot\Bigl(\frac{\theta_2}{2}\Big)\Big) \leq \frac12\Bigl(1+\frac{1 \cdot 2}{\theta_2}\Big),
    \end{align*}
    where, compared to \eqref{eq1:Riesz-Thorin non-symmetric} for $n=1$, we have gained the prefactor $\nicefrac{1}{2}$ at the expense of changing $(2n-1)$ to $2n$ in the numerator.
    Since $\sqrt{2} \geq \sqrt{\nicefrac{\pi}{2}} \geq \sqrt{\theta_2}$, this change makes log-concavity from the first step available and we can follow the argument in Case~1 up to \eqref{eq2:Riesz-Thorin non-symmetric} to give
    \begin{align*}
        \MR_p(A)
        &\leq  2^{\lambda-1} \Bigl( 1+\frac{1}{\theta_2} \bigl(1+\lambda \bigr)\bigl(2+\lambda\bigr) \Big).
    \end{align*}
    A priori, the right-hand side is different from how \eqref{eq2:Riesz-Thorin non-symmetric} reads when $n=1$. However, by convexity we have $2^{\lambda-1} \leq \nicefrac{1}{2}+ \nicefrac{\lambda}{2}$ and, consequently,
    \begin{align*}
        \MR_p(A) 
        &\leq 1+ \frac{1}{\theta_2} \bigl(1+\lambda \bigr)\bigl(1+\lambda\bigr)\bigl(1+\tfrac{\lambda}{2} \bigr) \\
        &\leq 1+ \frac{1}{\theta_2} \bigl(1+\lambda \bigr)\bigl(1+2 \lambda \bigr),
    \end{align*}
    which is precisely \eqref{eq2:Riesz-Thorin non-symmetric} for $n=1$. Now, we can complete the proof as in Case 1.
\end{proof}
In the special case that $A$ is a divergence form operator with real, bounded, elliptic coefficients and mixed boundary conditions as in the assumptions of Corollary~\ref{cor:Riesz-Thorin non-symmetric div}, we may choose $\theta_p$ such that we have the more precise estimate of \cite[Corollary~3.8]{Chill2021}:
\begin{align}
\label{Chillige Ungleichung}
\cot(\theta_p) \leq \frac{\sqrt{(p-2)^2 + p^2 \tan^2(\vartheta)}}{2 \sqrt{p-1}}.
\end{align}
\begin{proof}[Proof of Corollary~\ref{cor:Riesz-Thorin non-symmetric div}]
We abbreviate $c\coloneqq\sqrt{1+\tan^2(\vartheta)}$. For $n$ an integer, Theorem~\ref{thm:non-symm} together with \eqref{Chillige Ungleichung} yields the universal bound
    \begin{align*}
        \MR_{2n}(A) 
        &\leq 1+ \sqrt{2n-1} \sqrt{(n-1)^2 + n^2 \tan^2(\vartheta) } \\
        &\leq 1 + cn \sqrt{2n-1}.
    \end{align*}
We retain the notation from the proof of Corollary~\ref{cor:Riesz-Thorin non-symmetric}. The argument follows the same pattern, with two changes: We use $x \mapsto 1+cx^{3/2}$ as our log-concave function and we apply H\"older's inequality with conjugate exponents $\nicefrac32$ and $3$. No special treatment for $n=1$ is needed since this time we have log-concavity for all $x \geq 1$. In this vein, we get
    \begin{align*}
         \MR_p(A) 
         & \leq \Bigl( 1+ c n \sqrt{2n-1} \Big)^{1-\lambda}\Bigl( 1+ c (n+1)\sqrt{2n+1} \Big)^\lambda\\
         &\leq 1+ c\Bigl( (1-\lambda) \bigl(n\sqrt{2n-1}\bigr)^{2/3} + \lambda  \bigl((n+1)\sqrt{2n+1}\bigr)^{2/3} \Bigr)^{3/2} \\
         &= 1 + c \Bigl( (1-\lambda)^{2/3} n^{2/3} \cdot (1-\lambda)^{1/3} (2n-1)^{1/3} + \lambda^{2/3} (n+1)^{2/3} \cdot \lambda^{1/3} (2n+1)^{1/3} \Bigr)^{3/2} \\
         &\leq 1 + c \bigl(n + \lambda\bigr) \bigl(2 n+ 2 \lambda-1\bigr)^{1/2}.\\
         &=1 + \frac{c}{2} \Bigl( \frac{p}{n(n+1)}\Big)^{3/2} \bigl(n+\lambda\bigr)\bigl(n+1-\lambda \bigr) \bigl( n+\lambda - \tfrac{1}{2}\bigr)^{1/2} \bigl(n+1-\lambda\bigr)^{1/2} \\
         &\leq 1 + \frac{c}{2} \Bigl( \frac{p}{n(n+1)}\Big)^{3/2} \bigl(n + \tfrac12\bigr)^2\bigl(n+\tfrac14\bigr)\\
         &= 1+ \frac{c p^{3/2}}{2} \biggl( \frac{(n+\frac12)^{4} (n+\frac14)^{2}}{n^3(n+1)^3} \biggr)^{1/2}.
    \end{align*}
    It suffices to bound the remaining term in $n$ by $1$. While there may be more elegant ways to see this, expanding the numerator and denominator again does the job:
    \begin{align*}
        n^3(n+1)^3-\Bigl(n+\frac12\Bigr)^4\Bigl(n+\frac14\Bigr)^2
        =\frac{(n-1)(128n^4+240n^3+144n^2+40n+24)+23}{256} >0. & \qedhere
    \end{align*}
\end{proof}
\bibliographystyle{abbrv}
\bibliography{references.bib}

\end{document}